\documentclass[reqno,makeidx,11pt]{amsart}
\usepackage{amssymb,amsfonts,amsmath,graphicx
}
\def\N{{\mathbb{N}}}
\def\0{{\mathbb{O}}}

\def\Q{{\mathbb{Q}}}
\def\R{{\mathbb{R}}}

\def\Z{{\mathbb{Z}}}

\newcommand{\set}[1]{{\left\{{#1}\right\}}}

\newcommand{\actson}{\curvearrowright}

\newcommand{\Iso}{{\rm Iso}}

\def\fig{ \centerline{Fig. \the\count200\global\advance\count200 by 1}}
\newtheorem{thm}{Theorem}[section]

\newtheorem{lem}[thm]{Lemma}
\newtheorem{prop}[thm]{Proposition}

\newtheorem*{thm*}{Theorem}
\newtheorem*{lem*}{Lemma}
\newtheorem*{cor*}{Corollary}

\theoremstyle{definition}
\newtheorem{defn}[thm]{Definition}
\newtheorem{ex}[thm]{Example}

\theoremstyle{remark}
\newtheorem{rem}[thm]{Remark}

\title[Invariant proper metrics on coset spaces]
{Invariant proper metrics on coset spaces}

\author{Claire Anantharaman-Delaroche}
\address{Laboratoire de Math\'ematiques-Analyse, Probabilit\'es, Mod\'elisation-Orl\'eans (MAPMO - UMR6628),Ê
F\'ed\'eration Denis Poisson (FDP - FR2964),
CNRS/Universit\'e d'Orl\'eans,
B. P. 6759, F-45067 Orl\'eans Cedex 2}
\email{claire.anantharaman@univ-orleans.fr}



\begin{document}

\begin{abstract} It is known that for every second countable locally compact group $G$, there exists a proper $G$-invariant metric
which induces the topology of the group. This is no longer true for coset spaces $G/H$ viewed as $G$-spaces. We study
necessary and sufficient conditions which ensure the existence of such metrics on $G/H$.
\end{abstract}
\maketitle

\section*{Introduction} 

Given a Hausdorff topological group $G$, the study of the existence of a $G$-invariant\footnote{We let $G$ act on itself to the left.} metric on $G$ which defines its topology (in which case we say that the metric is compatible)   is an old problem which was solved by Birkhoff \cite{Bir} and Kakutani \cite{Kak}: a topological group is metrizable if and only if  it is Hausdorff and  the identity element $e$ of $G$ has a countable basis of neighbourhoods (i.e. $G$ is first countable). Moreover in this case the metric can be taken to be $G$-invariant (see for instance \cite[\S 3]{Bou1045} or \cite[\S 1.22]{M-Z}). More recently, in the last decade,  appeared the need for  $G$-invariant compatible metrics  which are proper in the sense that the balls are relatively compact. This is for instance the case when one wants to attack  the Baum-Connes
conjecture by considering proper affine isometric actions of $G$ on various Banach spaces, following an idea of Gromov
(see \cite{Tu}, \cite{H-P}). Actually, this problem had been solved in the seventies by Struble \cite{Struble}: a locally compact group $G$ has a proper  $G$-invariant compatible metric if and only if it is second countable.

In this paper, we are interested in the same questions on coset spaces.  Let $H$ be a closed subgroup of a Hausdorff topological group $G$. Whenever $G$ is metrizable, it is also the case for the quotient topological space $G/H$ (see \cite[\S 1.23]{M-Z}). However, even if $G$ is a nice second countable locally compact  group, there does not always exist a  $G$-invariant compatible metric on $G/H$ (see Example \ref{ex:neg}). Apart from the trivial case where $H$ is a normal subgroup of $G$, there is another  well-known exception: when $G$ is locally compact second countable and $H$ is  compact, it is easily seen, using Struble's theorem, that there is a proper $G$-invariant compatible metric on $G/H$ (Proposition \ref{compact}).

We extend this result to give a necessary and sufficient condition (Proposition \ref{CNS}) for the existence of such a metric in general. Several simple necessary conditions
for this existence follow immediately (Proposition \ref{prop:CN}): 
\begin{itemize}
\item there must exist  a non-zero relatively invariant positive measure on $G/H$ (i.e. the modular function of $H$ extends to a continuous homomorphism from $G$ to $\R_{*}^+$;
\item $H$ must be maximally almost periodic when the $G$-action on $G/H$ is effective\footnote{Without lost of generality, we can always make this assumption (see Remark \ref{rem:utile}).}; 
\item $H$ must be almost normal in $G$ in a topological sense (see Definition \ref{def:an}). 
\end{itemize}
However, these three conditions together are not enough in general to imply the existence of a proper  $G$-invariant compatible metric on $G/H$ (see Example \ref{ex:ex}). A 
noteworthy exception is the case where $G/H$ is discrete. Of course, the usual discrete metric on $G/H$ is $G$-invariant but it is not proper unless $G/H$
is finite. We show that  
 a proper ({\it i.e.} with finite balls) $G$-invariant compatible metric exists on the discrete coset space $G/H$ if and only if $H$ is an almost normal subgroup of $G$ (Theorem \ref{thm:discrete}).

 We also observe that in the opposite case where $G/H$ is connected, the existence of a $G$-invariant compatible metric and of a proper one are equivalent (Proposition \ref{prop:connected}) and it seems to be rather scarce (in non-trivial situations).

In Section 3, we consider briefly the case of a general $G$-space $X$. We show that for a group $G$ acting on a countable (discrete!) space $X$, there exists a proper $G$-invariant compatible metric if and only if the orbits of all the stabilizers of the action are finite. The proof uses a recent nice  theorem of Abels, Manoussos and Noskov \cite{A-M-N}, whose arduous proof  shows that 
whenever  a locally compact group $G$  acts properly on a second countable locally compact space $X$,
 there exists a proper $G$-invariant compatible metric on $X$.
 
 Of course, our proposition \ref{compact} in Section 2 is a direct consequence of this latter result from \cite{A-M-N}. However, it seemed to us more accessible to deduce this proposition \ref{compact} directly from its particular case which is the above mentioned theorem of Struble.
 
 In Section 1, we first introduce some notation and definitions. Since the notion of proper action plays a crucial role in our study, we also recall there, or prove, the related facts that we need in the sequel.

\section{Proper actions on metric spaces}
\subsection{Preliminaries} The topological spaces considered in this paper will always  be Hausdorff. We recall that a locally compact space $X$ is second countable if and only if it is metrizable and $\sigma$-compact (see for instance \cite[page 43]{Bou1045}). In particular, a locally compact  group is second countable if and only if it is
first countable and $\sigma$-compact.

Let $G$ be a topological group acting continuously on a locally compact space $X$. The action is said to be {\it proper} if the map $(g,x) \to (gx,x)$
from $G\times X$ into $X\times X$ is proper. This property is equivalent to the fact that  for every pair $A,B$
of compact subsets of $X$, the set $\set{g\in G: gA\cap B \not= \emptyset}$ is relatively compact, and also to the fact that for every pair $x,y$
of points of $X$ there exist neighbourhoods $V_x$ of $x$ and $V_y$ of $y$ such that $\set{g\in G: gV_x\cap V_y \not= \emptyset}$ is relatively compact

 Recall that  a topological group $G$  which acts properly on a locally compact space is itself locally compact.

The space of continuous functions from a topological space into itself will always be equipped with the compact-open topology.

Given a metric space $(X,d)$, $x\in X$ and $r>0$, we shall denote by $B(x,r)$ the closed ball $\set{y\in X: d(x,y) \leq r}$.
Whenever these balls are compact, we say the $(X,d)$ is a {\it proper metric space}\footnote{A proper and discrete metric space is also said to be {\it locally finite}.}. Complete riemannian manifolds and locally finite graphs with the geodesic distance
are such examples.

A metric is said to be discrete if it induces the discrete topology. We shall denote by $\delta$ the usual discrete metric such that $\delta(x,y) = 1$ whenever $x\not= y$.
Note that $(X,\delta)$ is a proper metric space only when $X$ is a finite set.

\subsection{Proper actions of $\Iso(X,d)$}

 Let $(X,d)$ be a metric space. We denote by $\Iso(X,d)$ its group of isometries, equipped with the compact-open topology. We recall that
this topology coincides with the topology of pointwise convergence (\cite[Th\'eor\`eme 1, page 29]{Bou1084}), that $\Iso(X,d)$ is a topological group
(\cite[Corollaire, page 48]{Bou1084}) and that $\Iso(X,d)$ acts continuously on $X$. However, even if $X$ is locally compact, the group $\Iso(X,d)$
is not always locally compact. It suffices for instance to consider the case of the group of bijections of the set $X = \N$ with its metric $\delta$.

We begin by reviewing some topological features which insure that $\Iso(X,d)$ is locally compact and acts properly on $X$.
The earliest  result is the following theorem due to Dantzig and van der Waerden \cite{D-W}.

\begin{thm}\label{D-W} Let $(X,d)$ be a locally compact, connected, metric space. The group $\Iso(X,d)$ of isometries of $X$ is locally compact, second countable and acts properly on $X$.
\end{thm}

\begin{proof}  For a proof, see \cite[Theorem 4.7]{K-N}. Since it is not explicitely shown in this reference that the action is proper, we give some
precisions, for the sake of completeness. We have to show that for every $x,y\in X$, there are neighbourhoods $V_x$ and $V_y$ of $x,y$ respectively
such that $\set{f\in \Iso(X,d) : f(V_x) \cap V_y \not=\emptyset}$ is relatively compact. We choose $r>0$ small enough such that the ball $B(y,2r)$ is  compact.
Since
$$\set{f\in \Iso(X,d) : f(B(x,r))\cap B(y,r)  \not=\emptyset} \subset\set{f\in \Iso(X,d) : d(f(x),y) \leq 2r}$$
it suffices to show that the right hand side, denoted by $E$, is relatively compact. Let $(f_n)$ be a sequence\footnote{The use of sequences is justified in
\cite[page 46]{K-N}.} in $E$. Passing, if necessary to a subsequence, we may assume that $(f_n(x))$ converges.  Then, by \cite[Lemma 3, page 47]{K-N} we may assume that $(f_n)$ converges pointwise,
and by \cite[Lemma 4, page 48]{K-N} the limit is an isometry.
\end{proof}

This theorem has been recently extended by Gao and Kechris. As a particular case, they obtain in \cite[Corollary 5.6]{G-K} the  following  result. Its proof is also contained in \cite[Theorem 3.2]{A-M-N}.

\begin{thm}\label{G-K} Let $(X,d)$ be a proper metric space. Then the group $\Iso(X,d)$ is locally compact and acts properly on $X$.
\end{thm}

This theorem applies in particular to the group of automorphisms of a locally finite graph $X$ equipped with its geodesic distance.

We shall also need the following result, that we have not met in the literature. Below, $X$ is a discrete space. We denote by ${\rm Map}(X)$ the space of maps from $X$ to $X$ endowed with the topology of pointwise convergence.
If $\delta$ is the usual discrete metric, note that $\Iso(X,\delta)$ is nothing else than the group ${\rm Bij}(X)$ of bijections from $X$ to $X$. It is not closed in ${\rm Map}(X)$ when the cardinal of $X$ is infinite.

\begin{thm}\label{moi} Let $G$ be a group acting (to the left) on a  space $X$. Let $\rho$ be the corresponding homomorphism from $G$ into ${\rm Bij}(X)$ and denote by
$G'$ the closure of $\rho(G)$ in ${\rm Map}(X)$. Then $G'$ is a subgroup of ${\rm Bij}(X)$ acting properly on the discrete space $X$ if and only if the orbits of all the stabilizers of
the $G$-action on $X$ are finite. Moreover, in this case the group $G'$ is locally compact and totally disconnected.
\end{thm}

\begin{proof} Obviously we may replace $G$ by $\rho(G)$, that is, we may assume that $G$
is a subgroup of ${\rm Bij}(X)$.

The condition on the orbits is of course necessary. Assume now that it is fulfilled. 
We first check that $G'$ is contained into ${\rm Bij}(X)$. Obviously,  the elements of $G'$ are injective maps, since $X$ is discrete. We now take $f\in G'$
and show that $f$ is surjective. Let $(g_i)$ be a net in $G$ such that $f = \lim_i g_i$  in the topology of the pointwise convergence. We fix $x\in X$ and $i_0$
such that $f(x) = g_i x$ for $i\geq i_0$. We set $g = g_{i_0}$. So $h_i = g^{-1}g_i$ is in the stabilizer $G_x$ of $x$ whenever $i\geq i_0$.
Observe that $\lim_i h_i = g^{-1}f$. In particular, $f_1 =  g^{-1}f$ is an injective map
and for $y\in X$, $h\in G_x$, we have $f_1(h y)=(g^{-1}f)(hy) = \lim_i (h_ih) y \subset G_x y$. Hence, $f_1(G_x y)\subset G_x y$. Since
$G_x y$ is a finite set, it follows that $y$ is in the range of $f_1$. Thus, we see that $f_1\in {\rm Bij(X)}$, and $f\in {\rm Bij(X)}$ too.

Now, we want to show that $G'$ is a locally compact group. For $x\in X$, we set $G'_{x} = \set{g'\in G' : g'(x) = x}$. It is easily checked that 
$G'_{x}$ is the closure of $G_x$. Let us prove that $G_x$ is relatively compact. This follows from the fact that $G_x$
is a subset of $\Pi_{y\in X} G_x y$, which is compact, by the Tychonov theorem.  Since the finite intersections of such stabilizers $G'_{x}$
form a basis of neighbourhoods of the identity in $G'$, we see that $G'$ is a locally compact group. Note that the
$G'_{x}$ are both compact and open, and so $G'$ is totally disconnected.

Finally, since $X$ is discrete and the $G'$-action has compact stabilizers, we see that the $G'$-action is proper.
\end{proof}

\section{Proper $G$-invariant compatible metrics on coset spaces}
Let $X$ be a topological space. A {\it compatible metric} on $X$ is a distance which defines the topology. 

\begin{defn} Let $G$ be a topological group, acting continuously to the left on a topological space $X$. A {\it compatible $G$-invariant metric} on $X$
is a compatible distance $d$ such that $d(gx,gy) = d(x,y)$ for every $g\in G$ and $x,y\in X$.
\end{defn}

Let us remind, for reference purpose,  the theorem of Struble mentioned in the introduction.

\begin{thm}[\cite{Struble}]\label{Strub} A locally compact group $G$ has a proper (left) $G$-invariant compatible metric if and only if it is second countable.
\end{thm}

We now consider a topological group $G$ and a closed subgroup $H$. Whenever $G$ is metrizable, the coset space $G/H$ is
still metrizable \cite[\S 1.23]{M-Z}. However, there is not always a $G$-invariant compatible metric, as shown by the following example.

\begin{ex}\label{ex:neg} Let $G = \R^+_{*}\ltimes \R$ and $H = \R^+_{*}$ viewed as a closed subgroup of $G$. Then $(a,b)\mapsto b$ is a homeomorphism
from $G/H$ onto $\R$. Under this identification, the left action of $a\in H$ on $x\in \R$ is the product in $\R$. If $d$ is a $G$-invariant metric on $G/H = \R$, we have
$d(0,x) = d(0,ax)$ for every $a\in H$ and $x\in \R$. If follows that $d$ cannot define the topology of $\R$.
\end{ex}

In case $H$ is a compact subgroup of a  metrizable topological group $G$,  there is still a $G$-invariant metric on $G/H$.
Moreover, if $G$ is in addition generated by a compact neighbourhood  of the identity (for instance if $G$ is locally compact, metrizable and connected)
it has been proved by Kristensen \cite{Kris} that there is even  a proper $G$-invariant compatible metric on $G/H$.

We first show that, using Theorem \ref{Strub}, it is not difficult to extend the above fact  to any locally compact second countable group.

\begin{prop}\label{compact} Let $G$ be a locally compact second countable group and $H$ a compact subgroup. Then there exists a proper $G$-invariant compatible
metric on $G/H$.
\end{prop}

\begin{proof} Let $d$ be a proper left $G$-invariant compatible metric on $G$. We set $\dot{g} = gH$ and define on $G/H$ the Hausdorff distance
\begin{align*}
\tilde{d}(\dot{g_1}, \dot{g_2}) &=  d(g_1H,g_2 H)\\
&= \max(\sup_{h_1\in H}\inf_{h_2\in H} d(g_1h_1,g_2h_2), \sup_{h_2\in H}\inf_{h_1\in H} d(g_1h_1,g_2h_2)).
\end{align*}

This distance is obviously $G$-invariant. Let us show that it is compatible. First, using the fact that $H$ is compact, given $\varepsilon >0$, there is a neighbourhood $W$
of the unit $e$ such that $d(h,gh) < \varepsilon$ for $h\in H$ and $g\in W$. Therefore, for $g\in W$, we have
$$\sup_{h_1\in H}\inf_{h_2\in H} d(h_1,gh_2) \leq \sup_{h\in H}d(h,gh) \leq \varepsilon,$$
and similarly $\sup_{h_2\in H}\inf_{h_1\in H} d(h_1,gh_2)\leq \varepsilon$. It follows that $\dot{g} \mapsto \tilde{d}(\dot{e},\dot{g})$
is continu\-ous (for the quotient topology) at $\dot e$. Since $\tilde d$ is invariant, we deduce that its open balls are open for the quotient topology.
Assuming for the moment that the balls $B(r) = \set{\dot{g}: \tilde{d}(\dot{e},\dot{g})\leq r}$ are compact for the quotient topology, let us show
that every open neighbourhood $V$ of $\dot e$ for the quotient topology contains an open ball. This is immediate since $\cap_n B(1/n)\cap \big((G/H)\setminus V\big) = \emptyset$.
Then we conclude that $\tilde d$ is compatible.

It remains to check that $B(r)$ is compact for $r>0$. Let $g\in G$ such that $\dot g \in B(r)$. For every $h_2\in H$ there exists $h_1\in H$
with $d(h_1,gh_2)\leq r$. Then we have
$$d(e,gh_2)\leq d(e,h_1)+d(h_1,gh_2)\leq c+r$$
where $c= \sup_{h\in H}d(e,h)$. It follows that $B(r)$ is contained in the image by the quotient map of the closed ball in $G$ with radius $c+r$, centered in $e$,
which is compact since $d$ is proper.
\end{proof}

\begin{rem}\label{rem:utile} Let $G$ be a topological group, $H$ a closed subgroup of $G$ and $L$  a normal closed subgroup  of $G$, contained in $H$.
Using the canonical homeomorphism from $G/H$ onto $(G/L)/(H/L)$, we  see that there exists
a proper $G$-invariant compatible metric on $G/H$ if and only if there exists a proper $G/L$-invariant compatible metric on $(G/L)/(H/L)$. 
In particular, in the previous proposition, it is enough to assume that $H$ contains a closed normal subgroup $L$ of $G$ such that $H/L$ is compact.

An interesting particular case is $L = \cap_{g\in G}gHg^{-1}$, the largest normal subgroup of $G$ contained in $H$, since the action of $G/L$ onto $(G/L)/(H/L)$ is effective.
\end{rem}

In the rest of this section $H$ is a closed subgroup of a locally compact group $G$ and $\rho$ denotes the natural homomorphism from $G$ into the group of
homeomorphisms of $G/H$. Together with the previous proposition, the following lemma will provide a characterization of the coset spaces which carry a proper $G$-invariant compatible metric.

\begin{lem}\label{proper} Let $G$ be a second countable locally compact group and $H$ a closed subgroup of $G$. We assume that there is a compatible $G$-invariant metric $d$ on $X= G/H$ and a  locally compact subgroup $G'$ of $\Iso(G/H,d)$, which contains $\rho(G)$ and is such that its action on
$G/H$ is proper. Then  
 there exists  a compact subgroup $H'$ of $G'$ such that
\begin{itemize}
\item[(a)] $\rho(H)\subset H'$;
\item[(b)] the map $\tilde \rho : G/H\to G'/H'$ induced by $\rho$ is a homeomorphism.
\end{itemize}
\end{lem}

\begin{proof}
  We denote by $H'$ the stabilizer of $\dot e$ for the action of $G'$ on $ X=G/H$.
  The action of $G'$ on $X$ is transitive and proper: it follows that $g' \mapsto g'\dot e$ induces a homeomorphism $\theta$ from $G'/H'$ onto $X$.
 We immediately check that $\tilde \rho = \theta^{-1}$.
 \end{proof}

\begin{prop}\label{CNS} Let $G$ be a second countable locally compact group 
and $H$ a closed subgroup of $G$. The following conditions are equivalent:
\begin{itemize}
\item[(i)] There exists a proper $G$-invariant compatible metric on $G/H$.
\item[(ii)] There exists a continuous homomorphism $\varphi$ from $G$ into a metrizable locally compact group $G'$ and a compact subgroup $H'$ of $G'$ such that
\begin{itemize}
\item[(a)] $\varphi(H)\subset H'$;
\item[(b)] the map $\tilde \varphi : G/H\to G'/H'$ induced by $\varphi$ is a homeomorphism.
\end{itemize}
\end{itemize}
\end{prop}

\begin{proof} (ii) $\Rightarrow$ (i). Assume that (ii) holds. Since $G'/H'$ is $\sigma$-compact and $H'$ is compact, we see that $G'$ is $\sigma$-compact.
 Being metrizable, it is second countable. Then, applying Proposition \ref{compact}, we see that there exists a proper $G'$-invariant compatible metric on $G'/H'$.
It gives a metric with the required properties on $G/H$.

(i) $\Rightarrow$ (ii). Let $d$ be a proper $G$-invariant compatible distance on $X=G/H$ and let $G'=\Iso(X,d)$ be the group of isometries of $X$,
endowed with the compact-open topology. We take $\varphi = \rho$. Since the metric is proper, by Theorem \ref{G-K},
we know that $G'$ is a locally compact group which acts properly on $X$. Then we apply the previous lemma.
Moreover, $G'$ is metrizable since $(X,d)$ is $\sigma$-compact. 
\end{proof}

\begin{rem}\label{rem:} (a) We observe that in the above statement we may assume that $G'$ is a locally compact group of homeomorphisms of $G/H$ and $\varphi = \rho$. We may also replace $G'$ by the closure of $\varphi(G)$, and thus assume that $\varphi(G)$ is dense in $G'$.

(b) The fact that $\tilde \varphi$ is bijective is equivalent to $G' = \varphi(G) H'$ together with $H= \varphi^{-1}(H')$.

(c) When the action of $G$ over $G/H$ is effective, we may take $\varphi$   to be injective.
\end{rem}

We deduce the following consequence when $G/H$ is connected.

\begin{prop}\label{prop:connected} Let $G$ be a second countable locally compact group and $H$ a closed subgroup of $G$. Assume that $G/H$ is connected. The two following conditions are
equivalent:
\begin{itemize}
\item[(i)] there exists a $G$-invariant compatible metric on $G/H$;
\item[(ii)] there exists a proper $G$-invariant compatible metric on $G/H$.
\end{itemize}
\end{prop}

\begin{proof} Let $d$ be a $G$-invariant compatible metric on $G/H$. It follows from Theorem \ref{D-W} that  the group $\Iso(G/H,d)$ is locally compact, second countable, and acts properly on $G/H$. Then we apply Lemma \ref{proper} and Proposition \ref{CNS}.
\end{proof}

In general, unless there is an obvious candidate for $(G', H', \varphi)$, Proposition \ref{CNS} is not easy to apply. However, we shall see that it can be used to find obstructions for
the existence of a proper $G$-invariant compatible metric on $G/H$. We need first to introduce or recall some definitions.

\begin{defn}\label{def:an} Let $H$ be a closed subgroup of a locally compact group $G$. We say that $H$ is {\it topologically almost normal} (resp. {\it almost normal}) in $G$ if for every compact
(resp. finite) subset $K$
of $G$ there exists a compact (resp. finite) subset $K'$ of $G$ such that $HK\subset K' H$. 
\end{defn}

In other terms, $H$ is  topologically almost normal  if and only if the $H$-orbits of compact  subsets of $G/H$
are relatively compact. It is almost normal if and only if the $H$-orbits of elements of $G/H$ are finite.
 Note that compact or co-compact subgroups are topologically almost normal, and that normal subgroups are of course both topologically almost normal
 and almost normal. Whenever $H$ is an open subgroup of $G$, then $H$ is almost normal if and only if it is topologically almost normal. 
 
 \begin{ex}\label{ex:alm_nor} Let $G$ be a locally compact group acting continuously by isometries on a locally finite metric space ({\it e.g.} a discrete group of automorphisms of a locally finite graph) and let $H$
 be the stabilizer of some element $x\in X$. Then $H$ is an open almost normal subgroup of $G$ (see Theorem \ref{thm:discrete}).
  
 Classical exemples of almost normal subgroups are plentiful: $SL_n(\Z)$ in $SL_n(\Q)$, $\Z \rtimes \set{1}$ in $\Q\rtimes \Q^{+}_*$,
 $\langle x\rangle$ in $BS(m,n) = \langle t,x : t^{-1}x^m t = x^n\rangle$, ....
 \end{ex}
 Recall the a group $G$ is {\it maximally almost periodic}  if there exists a continuous injective homomorphism from $G$ into a compact group.

In the next statement, we take $L = \ker\rho = \cap_{g\in G} gHg^{-1}$, where $\rho$ is the canonical homomorphism from $G$ into the group of homeomorphisms of $G/H$
(so that the $G/L$-action on $(G/L)/(H/L)$ is effective).

\begin{prop}\label{prop:CN} Let $G$ be a second countable locally compact group and $H$ a closed subgroup of $G$. We assume the the existence of a proper $G$-invariant compatible metric on $G/H$. Then
\begin{itemize}
\item[(i)] there exists a non-zero relatively $G$-invariant (positive Radon) measure on $G/H$;
\item[(ii)] $H/L$ is maximally almost periodic;
\item[(iii)] $H$ is topologically almost normal.
\end{itemize}
\end{prop}

\begin{proof} (i) is a consequence of Proposition \ref{CNS}, since there is a non-zero relatively $G'$-invariant measure on $G'/H'$ (see \cite[Corollaire 1, page 59]{Bou1306}).
Again by Proposition \ref{CNS}, we get a continuous injective homomorphism from $H/L$ into $H'$ and so $H/L$ is maximally almost periodic.
Finally, (iii) is obvious: let $d$ be a proper $G$-invariant compatible metric, $K$ a compact subset of $G/H$ and set $r = \sup_{k\in K}d(\dot e,k)$; then
$HK$ is contained in the ball of radius $r$ and center $\dot e$.
\end{proof}

Consider for instance a real connected semi-simple, non compact, Lie group $G$ and let $G = KAN$ be a Iwasawa decomposition of $G$. Then, there does not exist
a compatible $G$-invariant metric on $G/AN$. Indeed, $G/AN$ does not carry any relatively $G$-invariant (hence invariant) non-zero measure.

\begin{prop}\label{prop:map} Let $G$ be a connected locally compact group and $H$ a closed co-compact subgroup such that the the left action of $G$ over $G/H$ is effective. A necessary condition for the existence of a compatible $G$-invariant  metric on $G/H$ is that $G$ is isomorphic to a group of the form $K\times \R^n$ where $K$ is a compact group.
\end{prop}

\begin{proof} Assume that there is a $G$-invariant compatible metric on $G/H$. Since $G/H$ is compact, we see, with the notation of Proposition \ref{CNS}, that $G'$ is compact. Moreover we may assume that $\varphi : G\to G'$ is injective (see Remark \ref{rem:} (c)). Hence $G$ is maximally almost periodic and therefore of the form $K\times \R^n$ by a theorem of Freudenthal-Weil (see \cite[Th\'eor\`eme 16.4.6]{Dix}).
\end{proof}

\begin{ex}\label{ex:ex} Let $G$ be a connected locally compact second countable group and $H$ a discrete co-compact subgroup. We assume that $G$ has no discrete subgroup $N$ in its center  such that $G/N$ is maximally almost periodic.  Then, there is no $G$-invariant compatible metric on $G/H$. Otherwise, arguing as in the proof of the previous proposition, there will be a continuous homomorphism $\rho$ from $G$ into a compact group $G'$, whose kernel $L$ is
 a discrete normal subgroup of $G$, so contained in its center, because $G$ is connected. Moreover, $G/L$ is maximally almost periodic, since its embeds in $G'$, hence a contradiction.
 
 Let us specify the above example to show that the conditions (i), (ii) and (iii) of Proposition \ref{prop:CN} are not sufficient to insure the existence of a proper $G$-invariant compatible metric. Let $G= SO_0(2,3)$ be the connected component of the group of matrices in $SL(p+q,\R)$ leaving invariant the quadratic form 
$x_1^{2}+x_2^{2} - x_3^{2} -x_4^{2}-x_5^{2}$ and let $H$ be a co-compact lattice in $G$. Then $H$ is topologically almost normal, residually finite (by a theorem of Malcev, since it is a  finitely generated and linear group),
and there exists an invariant probability measure on $G/H$. However, as explained above, $G/H$ does not carry a $G$-invariant compatible metric.
\end{ex}

Non-trivial examples of closed subgroups of $G$ such that $G/H$ is connected and has a $G$-invariant compatible metric seem to be rather scarce. 
In case the subgroup $H$ is open in $G$ (for instance when $G$ is a discrete group), the situation is quite different. Of course, in this case Condition (i) of Proposition \ref{prop:CN}
is always satisfied. We shall see that Condition (iii) suffices to imply the existence of a proper $G$-invariant compatible metric on $G/H$.

\begin{thm}\label{thm:discrete} Let $G$ be a second countable locally compact group and $H$ a closed subgroup. The following conditions are equivalent:
\begin{itemize}
\item[(i)] $H$ is open and there exists a proper $G$-invariant compatible metric $d$ on $G/H$;
\item[(ii)] $H$ is a stabilizer for a continuous action of $G$ by isometries on a locally finite (discrete) metric space;
\item[(iii)] $H$ is an open almost normal subgroup of $G$.
\end{itemize}
\end{thm}

\begin{proof}
 (i) $\Rightarrow$ (ii) is obvious. 
 
 (ii) $\Rightarrow$ (iii). Let $G\actson (X,d)$ be a continuous action on a proper discrete metric space  such that $H$
is the stabilizer of some $x_0\in X$. Denote by $\psi$ the map $g\mapsto gx_0$. Let $k\in G$ and set $r = d(x_0,kx_0)$. Then
$\psi(Hk)$ is contained in the ball of radius $r$ and center $x_0$, which is finite. Therefore we have $\psi(Hk) = \set{g_1 x_0,\dots,g_n x_0}$ and so $Hk = \cup_{i=1}^n g_i H$.
Moreover $H$ is open since $X$ is discrete.

(iii) $\Rightarrow$ (i). We assume that $H$ is open and almost normal. So the orbits of the stabilizers of the $G$ action on $G/H$ are finite.
By Theorem  \ref{moi}, there exists a subgroup $G'$ of $\Iso(G/H,\delta)$, where $\delta$ is the usual discrete metric, which contains $\rho(G)$ and acts properly on $G/H$.
Then the conclusion follows from Lemma \ref{proper} and Proposition \ref{CNS}.
\end{proof}

\section{Proper $G$-invariant compatible metrics on $G$-spaces}
We now consider a locally compact group $G$ acting on a locally compact space $X$ and we are interested in the existence of  $G$-invariant compatible metrics on $X$.
In the lecture notes \cite{Kos}, the following result is proved:

\begin{thm} {\rm (\cite[Theorem 3, page 9]{Kos})} Let $G$ be a locally compact group acting properly on a second countable locally compact space $X$. 
Then there exists a  $G$-invariant compatible metric on $X$.
\end{thm}

Recently, this result has been improved in \cite{A-M-N}.

\begin{thm}[\cite{A-M-N}]\label{thm:AMN} Let $G$ be a locally compact group acting properly on a second countable locally compact topological space $X$.
Then there is a proper $G$-invariant compatible metric on $X$.
\end{thm}

Of course, the fact that the action is proper is not necessary to obtain the conclusion.

\begin{prop} Let $G$ be a group acting on a countable set $X$. The following conditions are equivalent:
\begin{itemize}
\item[(i)] there exists a $G$-invariant locally finite metric $d$ on $X$;
\item[(ii)] the orbits of all the stabilizers of the $G$-action are finite.
\end{itemize}
\end{prop}

\begin{proof} Denote by $G_x$ the stabilizer of $x$. We observe that (i) $\Rightarrow$ (ii) is immediate since for every $x,y\in X$, the set $G_x y$ is contained in the ball of center $x$ and radius $d(x,y)$.

(ii) $\Rightarrow$ (i). We follow the proof of (iii) $\Rightarrow$ (i) in Theorem \ref{thm:discrete}.  We introduce $\rho : G\to {\rm Bij}(X)$ and the group $G'$ as in Theorem \ref{moi}.
Since $G'$ is a locally compact group acting properly on $X$, the conclusion follows from Theorem \ref{thm:AMN}.
\end{proof}

%\bibliographystyle{plain}

%\bibliography{properbiblio}
\end{document}